\documentclass{article}
\usepackage[utf8]{inputenc}
\usepackage[T2A]{fontenc}
\usepackage[russian,english]{babel}

\usepackage[style=ieee, isbn=false, url=false]{biblatex}
\usepackage{graphicx} 
\usepackage{amsfonts, amsbsy, amssymb, amsthm, amsmath, amscd}

\newtheorem{theorem}{Theorem}
\newtheorem{lemma}{Lemma}
\newtheorem{remark}{Remark}

\title{Fractional Laplacian in {\sf V}-shaped waveguide}
\author{Fedor Bakharev\thanks{St.Petersburg State University, Universitetskaya emb. 7-9, St.Petersburg, 199034, Russia, e-mail: f.bakharev@spbu.ru}
\ and
Sergey Matveenko\thanks{Chebyshev Laboratory, St. Petersburg State University, 14th Line V.O., 29, Saint Petersburg 199178 Russia, e-mail: matveis239@gmail.com}
}

\begin{document}

\maketitle

\noindent{\bf Abstract.} The spectral properties of the restricted fractional Dirichlet Laplacian in {\sf V}-shaped waveguides are studied. The continuous spectrum for such domains with cylindrical outlets is known to occupy the ray $[\Lambda_\dagger, +\infty)$ with the threshold corresponding to the smallest eigenvalue of the cross-sectional problems. In this work the presence of a discrete spectrum at any junction angle is established along with the monotonic dependence of the discrete spectrum on the angle.

\medskip


\medskip

\noindent{\bf AMS classification codes:} Primary: 35R11, Secondary: 81Q10.

\section{Introduction}

The goal of this work is to study the spectral properties of the restricted fractional Dirichlet Laplacian in specific types of junctions of cylindrical domains --- {\sf V}-shaped waveguides.

The spectrum of the Dirichlet problem for the conventional Laplace operator in such domains is well-studied. The continuous spectrum in domains with cylindrical outlets to infinity always occupies the ray $[\Lambda_\dagger, +\infty)$ with the threshold coinciding with the smallest eigenvalue of the problems on the orthogonal cross-sections of the outlets (see, for instance, \cite{DaLaRay_2012}). In particular, in case of {\sf V}-shaped waveguide this threshold does not depend on the angle of the waveguide's opening. The research of the discrete spectrum started with the work \cite{ESS} where the planar case of the right-angled waveguide was considered. Later in \cite{Av91} it was proven that for small angles arbitrarily many eigenvalues could be obtained below $\Lambda_\dagger$. In \cite{DaLaRay_2012} the monotonic dependence of the eigenvalues and their number on the opening size of the waveguide was established. Another approach was performed in works \cite{na2010, na2011, na2012} where the asymptotics of the spectrum with angle close to straight angle was studied.  The multidimensional case does not differ much from the two-dimensional one. The only difference is in the variety of the possible cross-sections and the way of forming junctions. In some special cases it is possible to establish the uniqueness of the eigenvalue below the threshold in $\mathbb{R}^3$ (see e.g. \cite{BaMaNa}), however, some cross-sections could provide multiple eigenvalues even in case of right angle of the opening \cite{BaMaNaZaa} (the {\sf X}-shaped waveguides were studied, but for the {\sf V}-shaped waveguides similar result is true).

For the restricted fractional Dirichlet Laplacian, the structure of the continuous spectrum is known to be the same (see work \cite{BaNa2023JST}). In current work, we obtain results analogous to those described above: we establish the presence of a discrete spectrum at any angle of the junction, as well as the monotonic dependence of the discrete spectrum on the angle. However, up to now authors do not know how to prove the uniqueness of the eigenvalue below the threshold even for sufficiently large angles. 

The main difficulty of the problem from our point of view is the absence of the Dirichlet-Neumann bracketing technique which of a great use for the Laplace operator in works mentioned above. The main reason is the non-local nature of the operator which allows outlets somehow interact with each other. A crucial strategy is to transform the non-local operator to the local one involving the so-called Caffarelli-Silvestre extension (see original paper \cite{CaSi07}, earlier related work \cite{MO}, and description below). However, the cross-sections of the valuable domain in this case become unbounded and intersect with each other, so the methods used in aforementioned works are still out of use. Inspired by work \cite{Av91} we propose a transformation of the {\sf V}-shaped waveguide into a cylinder which successfully takes into account the unboundedness of cross-sections and allows to prove the existence of the discrete spectrum as well as the monotonic dependence on the opening angle.

The structure of the article is the following. In the Sec. \ref{notation} we introduce the notation, describe the geometry of the domain, provide necessary definitions and properties of the restricted fractional Dirichlet Laplacian, and give the necessary information on  the Caffarelli-Silvestre extension. In the Sec. \ref{existence} we prove the existence of the discrete spectrum for any angle. In the Sec. \ref{monotonicity} we prove that the discrete spectrum monotonically depends on the angle of the {\sf V}-shaped waveguide.

\section{Notation and statement of the problem}
\label{notation}

Let $\{e_1, e_2, \ldots, e_n\}$ represent an orthonormal basis in the Euclidean space ${\mathbb R}^n$, $n\geqslant 2$. We specify the first and last directions and denote by $\pmb{\bf x} = (x, y, z)$ the coordinates of the vector $xe_1 + \sum_{j=2}^{n-1} y_je_j + ze_n$, where $y=(y_2, y_3, \ldots, y_{n-1})$. In case $n=2$ we assume that the term $y$ is omitted. 

Let $\omega$ be a connected domain with Lipschitz boundary in ${\mathbb R}^{n-1}$.
The {\sf V}-shaped waveguide $\Omega_\beta$ we define by the formula
\begin{equation*} \label{def_waveguide}
\Omega_\beta = \{(x, y, z)\in\mathbb R^{n}\colon (x\sin\beta - |z|\cos\beta, y)\in\omega\}.
\end{equation*}
When $n=2$, if $\omega$ represents a unit segment, the waveguide $\Omega_\beta$ is a unit strip broken at an angle~$\beta$.  If $n>2$, the waveguide $\Omega_\beta$ consists of a junction of two semi-infinite cylinders with cross-section $\omega$, which meet at an angle $\beta$.

\begin{figure}[ht!]
\begin{center}
    \includegraphics[width=0.7\textwidth]{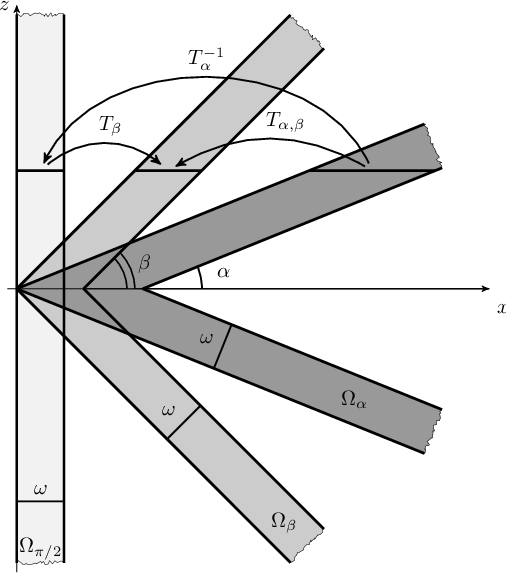}
    \caption{The {\sf V}-shaped waveguide $\Omega_\alpha$ (dark gray) overlaps the {\sf V}-shaped waveguide  $\Omega_\beta$ (light gray) and the straight tube $\Omega_{\pi/2}.$}
    \label{fig-01}
\end{center}
\end{figure}

For $s\in (0,1)$ we define the restricted Dirichlet fraction Laplacian $\mathcal{A}^\Omega_s$ by the quadratic form
\begin{equation*}
a_s^\Omega[u] = \int_{\mathbb R^n} |\pmb\xi|^{2s}|{\mathcal F}_n u(\pmb\xi)|^2 d\pmb\xi, \ \ \forall u\in\widetilde{H}^s(\Omega)
\end{equation*}
where $\mathcal{F}_n$ stands for the $n$-dimensional Fourier transform
\begin{equation*}
{\mathcal F}_n u(\pmb\xi) = \frac1{(2\pi)^{\frac n2}}\int_{\mathbb R^n}e^{-i\pmb\xi\cdot\pmb{\bf x}}u(\pmb{\bf x})\,d\pmb{\bf x},
\end{equation*}
where $\pmb\xi = (\xi, \eta, \zeta)$ is a corresponding frequency variable.  
The domain of the quadratic form $a_s^\Omega$ is defined as follows:
$$
\mathop{\rm Dom}\nolimits (a_s^\Omega)=\widetilde{H}^s(\Omega):=
\{u\in H^s(\mathbb{R}^n)\colon \mathop{\rm supp}\nolimits u\subset \overline{\Omega}\},
$$
where $H^s(\mathbb{R}^n)$ is the classical Sobolev--Slobodetskii space (see, e.g., \cite[Subsection 2.3.3]{Triebel})
$$
H^s(\mathbb{R}^n)=\{u\in L_2(\mathbb{R}^n)\colon |\pmb\xi|^s \mathcal{F}_n u(\pmb\xi)\in L_2(\mathbb{R}^n)\}.
$$

The connection between fractional differential operators and generalized harmonic extensions was established over fifty years ago \cite{MO} and gained popularity due to the influential work \cite{CaSi07}. Specifically, for a function $u$ belonging to $\widetilde{H}^s(\Omega)$, the function
\begin{equation}
    \label{def-Caf-Sil}
U(\pmb{\bf x}, t) = \int_{\mathbb R^n}{\mathcal P}_s(\pmb{\bf x} - \widetilde{\pmb{\bf x}}, t)u(\widetilde{\pmb{\bf x}})\,d\widetilde{\pmb{\bf x}}
\end{equation}
with the generalized Poisson kernel
\begin{equation}
    \label{def-Poisson-kernel}
{\mathcal P}_s(\pmb{\bf x}, t) = \frac{\Gamma\left(\frac{n+2s}2\right)}{\pi^{\frac  n2}\Gamma(s)}\frac{t^{2s}}{(|\pmb{\bf x}|^2 + t^2)^{\frac n2 + s}}
\end{equation}
is called the Caffarelli -- Silvestre extension of $u$. The function $U$ is a minimizer of the weighted Dirichlet integral 
\begin{equation}
    \label{def-Dir-int}
{\mathcal E}_s^\Omega(W) = \int_0^\infty\int_{\mathbb R^n}t^{1-2s}|\nabla W(\pmb{\bf x}, t)|^2d\pmb{\bf x} dt
\end{equation}
over the set 
\begin{equation*}
    {\mathcal W}_s^\Omega(u) = \{W = W(\pmb{\bf x}, t)\colon {\mathcal E}_s^\Omega(W) < \infty,\ W|_{t=0}=u\}.
\end{equation*}
Moreover, the Caffarelli -- Silvestre extension is a solution of the boundary problem
\begin{equation}
\label{def-Caf-Sil-prolem}
    - \mathop{\rm div}\nolimits t^{1-2s}\nabla U = 0, \qquad U(\pmb{\bf x}, 0) = u(\pmb{\bf x}).
\end{equation}
The following identity provide the Dirichlet fractional Laplacian of $u$ via its Caffarelli -- Silvestre extension
\begin{equation*}
    {\mathcal A}_s^\Omega u = - C_s\lim\limits_{t\to0+} t^{1-2s}\partial_t U(\cdot, t), \quad \text{where} \quad C_s = \frac{4^s\Gamma(s+1)}{2s\Gamma(1-s)}.
\end{equation*}
At  the same time the weighted Dirichlet integral is proportional to the quadratic form
\begin{equation}
    \label{def-form-via-energy-repr}
    a_s^\Omega[u]  = C_s{\mathcal E}_s^\Omega(U).
\end{equation}

In the degenerate case when $\beta=\pi/2$ it is proved in
\cite{BaNa2023JST} that the spectrum of ${\mathcal A}_s^{\Omega_{\pi/2}}$ is essential. Namely,
\begin{theorem}
\label{thm-sp-in-tube}
The spectrum of ${\mathcal A}_s^{\Omega_{\pi/2}}$ coincides with the ray
\[
\sigma({\mathcal A}_s^{\Omega_{\pi/2}}) = \sigma_{ess}({\mathcal A}_s^{\Omega_{\pi/2}}) = [\Lambda_\dagger,+\infty),
\]
where $\Lambda_\dagger=\lambda_1({\mathcal A}_s^\omega)$ is the first eigenvalue of the restricted Dirichlet fractional Laplacian on $\omega$.
\end{theorem}

Note that since the space $\widetilde{H}^s(\omega)$ is compactly embedded into $L_2(\omega)$,
the spectrum of ${\mathcal A}_s^\omega$ is purely discrete (see, e.g. \cite[Theorem 7.1]{NePaVa12} and \cite[Theorem 4.10.1]{BiSo}) and forms the unbound monotonically increased sequence. We denote by $\varphi$ the eigenfunction corresponding to the first eigenvalue $\Lambda_\dagger=\lambda_1({\mathcal A}_s^\omega)$. 

We define isomorphism between the straight tube $\Omega_{\pi/2}$ and the broken waveguide $\Omega_\beta$
(see fig. \ref{fig-01})
\begin{equation*}
T_\beta\colon \Omega_{\pi/2}\to\Omega_\beta,\qquad
T_\beta(\pmb{\bf x})=(x\csc\beta+|z|\cot\beta, y, z),    
\end{equation*}
and its inverse
\begin{equation*}
 T_\beta^{-1}\colon \Omega_\beta\to\Omega_{\pi/2},\qquad
T_\beta^{-1}(\pmb{\bf x})=(x\sin\beta-|z|\cos\beta, y, z).
\end{equation*}
Note that being Lipschitz these isomorphism preserve the Sobolev--Slobodetskii spaces thus the following lemma is true.
\begin{lemma}
\label{lem-Hs-iso}
    Given $s\in(0, 1)$ and $\beta\in(0,\pi/2)$, the coordinate transform $T_\beta$ provides an isomorphism between $\widetilde{H}^s(\Omega_\beta)$ and $\widetilde{H}^s(\Omega_{\pi/2})$, i.e. $u\in\widetilde{H}^s(\Omega_\beta)$ implies that $u \circ T_\beta\in \widetilde{H}^s(\Omega_{\pi/2})$ and $u\in\widetilde{H}^s(\Omega_{\pi/2})$ implies that $u \circ T_\beta^{-1}\in \widetilde{H}^s(\Omega_\beta)$.
\end{lemma}

For our purpose we need the following embedding statement for the Caffarelli -- Silvestre extension (see e.g. \cite{BaNa2023JST}).

\begin{lemma}
\label{lem-CS-L2-converge}
Let $n > 2 - 2s$. Suppose that $\Omega$
 is a bounded domain in $\mathbb R^n$; then the Caffarelli -- Silvestre extension of any function $u\in{H}^s(\Omega)$ belongs to $L_2(\mathbb R^n\times \mathbb R_+)$ with weight $t^{1-2s}$.
\end{lemma}

Hereinafter, we denote by $C$ various positive constants.

\section{Existence of the discrete spectrum}
\label{existence}

\begin{theorem}
    Given $s\in(0, 1)$ and $\beta\in(0,\pi/2)$, the discrete spectrum of ${\mathcal A}_s^{\Omega_\beta}$ is not empty.
\end{theorem}
\begin{proof}
Using the max-min principle, we require a function 
$u\in\widetilde{H}^s(\Omega_\beta)$ such that the following inequality holds:
\begin{equation*}
a_s^{\Omega_\beta}[u] - \Lambda_\dagger\|u;L_2(\mathbb R^n)\|^2<0.
\end{equation*}
Given that the Caffarelli-Silvestre extension is the minimizer of the weighted Dirichlet integral \eqref{def-Dir-int}, and based on \eqref{def-form-via-energy-repr}, our goal is to identify a function $\Psi\in{\mathcal W}_s^{\Omega_\beta}(u)$ for some $u\in\widetilde{H}^s(\Omega_\beta)$ such that 
\begin{equation}
\label{thm1-main-enq}
C_s{\mathcal E}_s^{\Omega_\beta}(\Psi) - \Lambda_\dagger\|\Psi(\cdot, 0); L_2(\mathbb R^n)\|^2 < 0.
\end{equation}
To construct the desired function we spread the first eigenfunction $\varphi$ of~${\mathcal A}_s^\omega$  along the {\sf V}-shaped waveguide as detailed below. The Caffarelli-Silvestre extension of $\varphi$ is denoted by $\Phi=\Phi(x, y, t)$.
Here and below we denote by $\widetilde T_\beta$ and $\widetilde T_\beta^{-1}$ the mappings of $\mathbb R^n\times\mathbb R_+$, given by 
$\widetilde T_\beta(\pmb{\bf x}, t) = (T_\beta(\pmb{\bf x}), t)$ and $\widetilde T_\beta^{-1}(\pmb{\bf x}, t) = (T_\beta^{-1}(\pmb{\bf x}), t)$.
Straightforward computation reveals that when $z\neq 0 $ the function
$\widetilde\Phi(x, y, z, t) = \Phi(T_\beta^{-1}(\pmb{\bf x}), t)$ satisfies the equation
\begin{equation}
\label{thm1-div-Phi-eq}
    -\mathop{\rm div}\nolimits\left(t^{1-2s}\nabla\widetilde\Phi(\pmb{\bf x}, t)\right) =  -\mathop{\rm div}\nolimits\left(t^{1-2s}\nabla\Phi(x, y, t)\right) = 0.
\end{equation}

We now introduce a family of test functions
\begin{equation*}
    \Psi(\pmb{\bf x}, t) = \chi_R(z, t)\widetilde\Phi(\pmb{\bf x}, t) + \varepsilon V(\pmb{\bf x}, t).
\end{equation*}
Here, $V\in C_0^\infty(\mathbb R^n\times\mathbb R_+)$ is a smooth perturbation that is even with respect to $z$ and has compact support. The cut-off multiplier 
$\chi_R$ is defined by the following formula
\begin{equation*}
    \chi_R(z, t) = \begin{cases}
        \chi(R^{-1}\sqrt{z^2 + t^2}),\quad &n -1 \leqslant 2 - 2s;\\
        \chi(R^{-1}|z|),\quad&n -1 > 2 - 2s,
    \end{cases}
\end{equation*}
with $\chi = \chi(z)$ being a monotone smooth cut-off function: it is equal to one for $z < 1$, equal to zero for $z > 2$. It is worth noting that $\Psi$ is not square summable without the multiplier $\chi_R$. We assume that $R$ is large enough and $\mathop{\rm supp}\nolimits V\subset \{\chi_R=1\}$.

The values for the large constant $R$ and the small parameter $\varepsilon$ will be specified later.
 
We substitute these test functions into the weighted energy integral \eqref{def-Dir-int} and get
\begin{equation*}
{\mathcal E}_s^{\Omega_\beta}(\Psi)  =
\int_0^\infty\int_{\mathbb R^n}t^{1-2s}\left(|\nabla(\chi_R\widetilde\Phi)|^2 + 2\varepsilon \nabla(\chi_R\widetilde\Phi) \cdot \nabla V + \varepsilon^2|\nabla V|^2\right)d\pmb{\bf x} dt.
\end{equation*}

Simple calculations show that 
\begin{align*}
|\nabla(\chi_R\widetilde\Phi)|^2 = 
\chi_R^2| \nabla\Phi\circ\widetilde T_\beta^{-1}|^2
+2\chi_R\partial_t\chi_R(\Phi\circ \widetilde T_\beta^{-1})(\partial_t\Phi\circ \widetilde T_\beta^{-1})
\\ +|\nabla\chi_R|^2\Phi^2\circ\widetilde T_\beta^{-1}
-2\cos\beta\chi_R\partial_z\chi_R(\Phi\circ \widetilde T_\beta^{-1})(\partial_x\Phi\circ \widetilde T_\beta^{-1}).
\end{align*}
and thus after changing the variable $\pmb{\bf x}\to T_\beta(\pmb{\bf x})$ the expression on the left-hand side of \eqref{thm1-main-enq} splits into the sum
\begin{equation}
\label{thm1-qform-as-sum}
C_s{\mathcal E}_s^{\Omega_\beta}(\Psi) - \Lambda_\dagger\|\Psi(\cdot,0); L_2(\mathbb R^n)\|^2 = 2(I_{1} + I_{21} + I_{22} +I_{23} + \varepsilon I_3 + \varepsilon^2 I_4), 
\end{equation}
where
\begin{align*}
    I_{1} &= \frac{C_s}{\sin\beta} \int_0^\infty t^{1-2s}\!\!\int_{\mathbb R^{n-1}\times\mathbb R_+} \chi^2_R(z,t) 
    |\nabla \Phi(x, y, t)|^2 \,d\pmb{\bf x} dt \\ &- 
    \frac{\Lambda_\dagger}{\sin\beta}\!\!\int_{\mathbb R^{n-1}\times\mathbb R_+} \chi^2_R(z,0) 
    |\Phi(x,y, 0)|^2 \,d\pmb{\bf x},\\
    I_{21} &= \frac{C_s}{2\sin\beta}\int_0^\infty\int_0^\infty t^{1-2s} \partial_t \chi_R^2(z, t) \int_{\mathbb R^{n-1}}\partial_t \Phi^2 (x, y, t)\, dxdydtdz,\\
    I_{22} &= \frac{C_s}{\sin\beta}\int_0^\infty\int_0^\infty t^{1-2s} |\nabla\chi_R(z, t)|^2 \int_{\mathbb R^{n-1}}\Phi^2(x, y, t) \,dxdydtdz,\\
    I_{23} &= -\frac{C_s \cot\beta}{2\sin\beta}\int_0^\infty\int_0^\infty t^{1-2s} \partial_z\chi_R^2(z, t) \int_{\mathbb R^{n-1}}\partial_x \Phi^2(x, y, t)\, dxdydtdz,\\
    I_{3} &= {C_s} \int_0^\infty t^{1-2s}\!\!\int_{\mathbb R^{n-1}\times\mathbb R_+} 
    \nabla \widetilde\Phi(\pmb{\bf x}, t) \cdot \nabla V(\pmb{\bf x}, t)\, d\pmb{\bf x} dt,\\
    I_{4} &=  \int_0^\infty \!\!\int_{\mathbb R^{n-1}\times\mathbb R_+} V^2(\pmb{\bf x}, t)\, d\pmb{\bf x} dt.
\end{align*}
Multiplier 2 stands on the right-hand side of \eqref{thm1-qform-as-sum} because $\Psi$ is an even function with respect to $z$, and the waveguide $\Omega_\beta$ is symmetric with respect to the plane $\{z=0\}$, so doubled integration in half-space gives the same result. Let us treat the terms separately. 

\noindent{\bf Estimate of $I_1$.} Since 
\begin{equation}
\label{chi-ineq}
\chi_R(z,t)\leqslant\chi_R(z,0),
\end{equation}
applying the Fubini theorem to the integrals in~$I_1$, we have the estimate
\begin{equation*}
\label{thm1-I_1-term-est}
    I_{1} \leqslant \frac{1}{\sin\beta} \int_{\mathbb R_+} \chi_R^2(z, 0)\left( C_s{\mathcal E}_s^{\omega}(\Phi) - \Lambda_\dagger\|\Phi(\cdot, 0); L_2(\mathbb R^{n-1})\|^2\right)\, dz = 0.
\end{equation*}
Due to \eqref{def-form-via-energy-repr}, and since $\varphi$ is an eigenfunction of ${\mathcal A}_s^\omega$, the expression in the brackets is equal to zero.
It should be mentioned that the inequality \eqref{chi-ineq} becomes an equality for $n>2$.

\noindent{\bf Estimate of $I_{21}+I_{22}$ in case $n-1>2-2s$.} 
In this case the term $I_{21} = 0$ because $\chi_R$ does not depend on $t$. Due to the Lemma \ref{lem-CS-L2-converge}, the remaining term allows the estimate
\begin{equation*}
    I_{22} \leqslant C\int_0^\infty |\partial_z\chi_R(z)|^2\, dz\leqslant CR^{-2}.
\end{equation*}

\noindent{\bf Estimate of $I_{21}+I_{22}$ in case $n-1< 2-2s$.} 
This case is possible only if $n=2$ and $s\in (0,1/2)$ and it is much more tricky. According to~\eqref{def-Caf-Sil} we obtain
$\Phi(\cdot, t) = {\mathcal P}_s(\cdot, t) * \varphi$. The Young convolution inequality provides the estimate
\begin{equation*}
    \|\Phi(\cdot, t); L_2(\mathbb R)\|^2 = \|{\mathcal P}_s(\cdot, t) * \varphi; L_2(\mathbb R)\|^2\leqslant \|{\mathcal P}_s(\cdot, t); L_2(\mathbb R)\|^2  \|\varphi; L_1(\mathbb R)\|^2,
\end{equation*}
and note that $\varphi$ is summable because it is square summable and has a compact support. Introducing a new variable $\widetilde{x}=x/t$ we write 
\begin{equation*}
     \|{\mathcal P}_s(\cdot, t); L_2(\mathbb R)\|^2  = \int_\mathbb R \frac{t^{4s}\,dx}{(x^2 + t^2)^{1+2s}}= \frac{1}{t}\int_\mathbb R \frac{d \widetilde{x}}{(1 +  \widetilde{x}^2)^{1+2s}},
\end{equation*}
and thus 
\begin{equation}
\label{thm1-CS-phi-norm-est}
    \|\Phi(\cdot, t); L_2(\mathbb R)\|^2 \leqslant C t^{-1}.
\end{equation}
Integrating by parts in the term $I_{21}$, we come to the formula
\begin{equation*}
    I_{21} = -\frac{C_s}{2\sin\beta}\int_0^\infty\int_0^\infty \partial_t\left(t^{1-2s} \partial_t\chi_R^2(z, t)\right) \int_{\mathbb R}\Phi^2(x, t)\,dxdtdz.
\end{equation*}
taking into account \eqref{thm1-CS-phi-norm-est}, we obtain an estimate
\begin{equation*}
    I_{21} + I_{22} \leqslant C\int_0^\infty\int_0^\infty \bigl| t^{-1}\partial_t\left(t^{1-2s} \partial_t\chi_R^2(z, t)\right)\bigr| + t^{-2s} |\nabla \chi_R(z, t)|^2\, dtdz.
\end{equation*}
After passing to the polar coordinates $t = r\sin\alpha$, $z = r\cos\alpha$, we get
\begin{equation*}
    I_{21} + I_{22} \leqslant C\int_R^{2R} R^{-1}r^{-2s} + R^{-2}r^{1-2s}\,dr \int_0^{\pi/2} (\sin\alpha)^{-2s} + (\sin\alpha)^{2-2s} d\alpha,
\end{equation*}
and since $s < 1/2$ the last integral converges. Hence,
\begin{equation*}
    I_{21} + I_{22} \leqslant CR^{-2s}.
\end{equation*}

\noindent{\bf Estimate of $I_{21}+I_{22}$ in case $n-1=2-2s$.}
In this case we have $n=2$ and $s=1/2$. Integrating by parts in the term $I_{21}$, we get
\begin{multline*}
    I_{21} = -\frac{C_s}{2\sin\beta}\int_0^\infty\int_0^\infty \partial_t^2\chi_R^2(z, t) \int_{\mathbb R}\Phi^2(x, t)\,dxdtdz\\
    - \frac{C_s}{2\sin\beta}\int_0^\infty \partial_t \chi_R^2(z, 0) \int_{\mathbb R^{n-1}}\Phi^2 (x, 0)\, dxdz.
\end{multline*}
Using \eqref{def-Caf-Sil-prolem}, we rearrange the last term, and thus,
\begin{multline}
\label{eq:I_12+I_22}
    I_{21} + I_{22}= \frac{C_s}{\sin\beta}\int_0^\infty\int_0^\infty |\partial_z\chi_R(z,t)|^2-\chi_R(z, t)\partial_t^2\chi_R(z, t) \int_{\mathbb R}\Phi^2(x, t)\,dxdtdz\\
    - \frac{C_s}{2\sin\beta}\int_0^\infty \partial_t \chi_R^2(z, 0) dz.
\end{multline}
Since $\Phi(x, t) = {\mathcal P}(\cdot, t) * \varphi(x)$ and the Fourier transform preserves $L_2$-norm, the following identity is fulfilled
\begin{equation*}
\|\Phi; L_2(\mathbb R\times(0, 2R))\|^2 = 
\int_0^{2R}\int_{-\infty}^\infty |{\mathcal F}_1{\mathcal P}_{1/2}(\cdot, t)(\xi)|^2|{\mathcal F}_1\varphi(\xi)|^2\,d\xi\,dt.
\end{equation*}
Direct computations show that
\begin{equation*}
{\mathcal F}_1{\mathcal P}_{1/2}(\cdot, t)(\xi) = \sqrt{\frac{\pi}2}e^{-t|\xi|},
\end{equation*}
thus, we have
\begin{equation*}
\|\Phi; L_2(\mathbb R\times(0, 2R))\|^2 =  \sqrt{\frac{\pi}2}
\int_0^{2R}\int_{-\infty}^\infty e^{-2t|\xi|}|{\mathcal F}_1\varphi(\xi)|^2\,d\xi\,dt.
\end{equation*}
Applying the Fubini theorem, we get
\begin{equation*}
\|\Phi; L_2(\mathbb R\times(0, 2R))\|^2 = \sqrt{\frac{\pi}2}
\int_{-\infty}^\infty \frac{1 - e^{-4R|\xi|}}{2|\xi|}|{\mathcal F}_1\varphi(\xi)|^2\,d\xi\,dt.
\end{equation*}
We decompose the last integral into a sum
\begin{equation*}
J_1 + J_2 + J_3 :=
\left(\int_{|\xi|\leqslant R^{-1}} + \int_{ R^{-1} \leqslant |\xi| \leqslant 1} + \int_{|\xi| \geqslant 1}\right) \frac{1 - e^{-4R|\xi|}}{2|\xi|}|{\mathcal F}_1\varphi(\xi)|^2\,d\xi\,dt.
\end{equation*}
The function $\varphi$ has a compact support, so its Fourier transform is a smooth function. Introducing a new variable $\widetilde\xi = R\xi$, we produce an estimate for $J_1$ 
\begin{equation*}
J_1 \leqslant C\int_{0}^1\frac{1-e^{-4\widetilde\xi}}{\widetilde\xi}d\widetilde\xi,    
\end{equation*}
with a convergent integral on the right-hand side. The other summands fulfilled the inequalities
\begin{equation*}
\begin{aligned}
J_2&\leqslant C\int_{R^{-1}}^1\frac{d\xi}{\xi}\leqslant C\log(R),\\
J_3&\leqslant C\int_1^\infty |{\mathcal F}_1\varphi(\xi)|^2\,d\xi\leqslant C.
\end{aligned}
\end{equation*}
Adding up the estimates for $J_1$, $J_2$, and $J_3$, we obtain
\begin{equation*}
\|\Phi; L_2(\mathbb R\times(0, 2R))\|^2 \leqslant C\log(R).
\end{equation*}
Noticing that the area $\{(z, t)\colon |z| < 2R,\ |t|<2R\}$ contains the support of $|\partial_z\chi_R|^2-\chi_R\partial_t^2\chi_R$, we apply the last estimate to \eqref{eq:I_12+I_22} and get
\begin{equation*}
    I_{21} + I_{22} \leqslant CR^{-1}\log(R).
\end{equation*}

\noindent{\bf Estimate of $I_{23}$.} 
The Cauchy -- Schwartz inequality guaranties that $I_{23}$ is finite since
\begin{multline*}
\frac{C_s^2}{\sin^2\beta}\left(\int_0^\infty\int_0^\infty t^{1-2s} |\partial_z\chi_R^2(z, t)| \int_{\mathbb R^{n-1}}|\partial_x \Phi^2(x, y, t)|\, dxdydtdz\right)^2 \\ \leqslant  
I_{22}\frac{C_s}{\sin\beta} \int_{\mathbb R_+} \chi_R^2(z, t){\mathcal E}_s^{\omega}(\Phi)\, dz.
\end{multline*}
Hence, due to the Fubini theorem the integrand in  $I_{23}$ is finite for almost all $z, t\in\mathbb R_+$, $y\in {\mathbb R}^{n-2}$ and
we have $I_{23} = 0$, since $\Phi$ is continuous when $t>0$ and 
the fundamental theorem of calculus results in 
\begin{equation*}
\int_{-\infty}^{\infty}\partial_x \Phi^2 (x, y, t)\, dx = 0.
\end{equation*}

\noindent{\bf Estimate of $I_3$.} 
Using \eqref{thm1-div-Phi-eq}, we integrate by parts the term $I_3$ and get
\begin{equation*}
    I_3 = -C_s\int_0^{\infty} \int_{\mathbb R^{n-1}} t^{1 - 2s} \partial_z \widetilde\Phi(x, y, 0, t) V(x, y, 0, t) \,dxdydt.
\end{equation*}
We point out that 
$\partial_z \widetilde\Phi(x, y, 0, t) \not\equiv 0$. Assuming the converse, we have
$$\partial_z \widetilde\Phi(x, y, 0, t) = -\cos\beta\,\partial_x\Phi(x \sin\beta, y, t)\equiv 0,$$ and therefore $\Phi(x,y,t)$ does not depend on $x$. But it leads to a contradiction because due to \eqref{def-Caf-Sil} and \eqref{def-Poisson-kernel} we have $\Phi(x,y,t)\to 0$ as $|x|\to\infty$. Thus, there exists a function $V$ such that $I_3<0$.

\noindent{\bf Estimate of $I_4$.}  
The last term $I_4$ in \eqref{thm1-qform-as-sum} is bounded.

\medskip

Finally, we combine all the estimates and obtain that
\begin{equation*}
C_s{\mathcal E}_s^{\Omega_\beta}(\Psi) - \Lambda_\dagger\|\Psi(\cdot,0); L_2(\mathbb R^n)\|^2 \leqslant 2\varepsilon I_3 + C(R^{-2\gamma}\log(R) + \varepsilon^2),
\end{equation*}
where $\gamma=1$ for $n-1\leqslant 2-s$ and $\gamma=s$ otherwise.
To end the proof, it remains to choose $\varepsilon$ to be equal $R^{-\gamma}$, then
the required inequality \eqref{thm1-main-enq} is satisfied for sufficiently large $R$.
\end{proof}

\section{Monotonicity of the discrete spectrum}
\label{monotonicity}

Let $0<\alpha<\beta<\pi/2$. We define a continuous piecewise-linear bijective map $T_{\alpha,\beta} = T_\beta\circ T_\alpha^{-1}:\Omega_\alpha\to \Omega_\beta$ (see fig. \ref{fig-01}). As above we define its extension $\widetilde T_{\alpha,\beta}: \Omega_\alpha\times\mathbb R_+\to \Omega_\beta\times\mathbb R_+$  by the formula $\widetilde T_{\alpha,\beta}(\pmb{\bf x}, t) = (T_{\alpha,\beta}(\pmb{\bf x}),t)$. 

\begin{lemma} \label{lem-q-form-monotony}
    Suppose there exist $\beta\in(0,\pi/2)$ and $u\in\widetilde{H}^s(\Omega_\beta)$ normalized in $L_2(\Omega_\beta)$  such that $a_s^{\Omega_\beta}[u]<\Lambda_\dagger$; then for all $\alpha\in(0,\beta)$
    functions $v_\alpha \in\widetilde{H}^s(\Omega_\alpha)$ defined as
    \begin{equation*}
v_\alpha = \left(\frac{\sin\alpha}{\sin\beta}\right)^{1/2}u \circ T_{\alpha,\beta}
\end{equation*}
are also normalized in  $L_2(\Omega_\alpha)$ and the map $\alpha \mapsto a_s^{\Omega_\alpha}[v_\alpha]$ increases monotonically, implying that $a_s^{\Omega_\alpha}[v_\alpha] < \Lambda_\dagger$ for all $\alpha \in (0, \beta)$.
\end{lemma}
\begin{proof} 
The equality $\|v_\alpha, L_2(\Omega_\alpha)\| = 1$ follows immediately from the fact that the Jacobian of $T_{\alpha,\beta}$ is constant and equals to $\frac{\sin\alpha}{\sin\beta}$.
Due to Lemma \ref{lem-Hs-iso}, the functions $v_{\alpha}$ belong to the space $\widetilde{H}^s(\Omega_\alpha)$. 

Now, let us denote by $U$ the Caffarelli -- Silvestre extension of $u$. 
Moreover, we define the family of functions
\begin{equation*}
    V_\alpha =  \left(\frac{\sin\alpha}{\sin\beta}\right)^{1/2} U \circ \widetilde T_{\alpha,\beta}
\end{equation*}
with traces $V_\alpha(\cdot, 0) = v_\alpha$.

Due to \eqref{def-form-via-energy-repr}
\begin{equation*}
    a_s^{\Omega_\beta}[u] = C_s{\mathcal E}_s^{\Omega_\beta}(U).
\end{equation*}
 Using the change of variables $\pmb{\bf x}\to T_\beta(\pmb{\bf x})$ we get
 \begin{equation*}
    a_s^{\Omega_\beta}[u] = C_s{\mathcal E}_s^{\Omega_{\pi/2}}(V_{\pi/2}) + \cos\beta\cdot {\mathfrak r}(V_{\pi/2})
\end{equation*}
where
\begin{equation*}
 {\mathfrak r} (V) = 
- 2 C_s\int_0^\infty\int_{\mathbb R^n}t^{1-2s} \mathop{\rm sgn}\nolimits(z) \partial_x V(\pmb{\bf x},t)\partial_z V(\pmb{\bf x},t)d\pmb{\bf x} dt.
\end{equation*}
According to Theorem~\ref{thm-sp-in-tube}, the max-min principle implies that 
$
C_s{\mathcal E}_s^{\Omega_{\pi/2}}(V_{\pi/2}) \geqslant \Lambda_\dagger.
$
Therefore, in case $a_s^{\Omega_\beta}[u]<\Lambda_\dagger$ it turns out that the term ${\mathfrak r}(V_{\pi/2})$ must be negative.
Hence, the following inequality holds true
 \begin{equation*}
    a_s^{\Omega_\beta}[u] > C_s{\mathcal E}_s^{\Omega_{\pi/2}}(V_{\pi/2}) + \cos\alpha\cdot {\mathfrak r}(V_{\pi/2}).
\end{equation*}
for all $\alpha\in(0, \beta)$. Changing  the variables $\pmb{\bf x}\to T_\alpha^{-1}(\pmb{\bf x})$ in the right hand side we rearrange the inequality to get
\begin{equation*}
    a_s^{\Omega_\beta}[u] > C_s{\mathcal E}_s^{\Omega_\alpha}(V_{\alpha}).
\end{equation*}
Note that, 
the  function $V_{\alpha}(\cdot, 0)$ 
coincides width $v_\alpha$ and hence belong to the space
${\mathcal W}_s^{\Omega_\alpha}(v_\alpha)$.
After obtaining the infimum  of the right hand side of the last inequality over the set ${\mathcal W}_s^{\Omega_\alpha}(v_\alpha)$, due to \eqref{def-form-via-energy-repr} we have $a_s^{\Omega_\beta}[u] > a_s^{\Omega_\alpha}[v_\alpha]$ which ends the proof.
\end{proof}

\begin{theorem} \label{thm-monotonicity}
    Given $s\in(0, 1)$, if ${\mathcal A}_s^{\Omega_\beta}$ has $k$ eigenvalues below the threshold for some $\beta\in(0,\pi/2)$, then the operator ${\mathcal A}_s^{\Omega_\alpha}$ has at least $k$ eigenvalues below the threshold for all $\alpha\in (0, \beta)$. Moreover, the $j$-th eigenvalue $\lambda_j({\mathcal A}_s^{\Omega_\alpha})$ is a monotonically increasing function with respect to $\alpha$ for all $1\leqslant j\leqslant k$. 
\end{theorem}
\begin{proof} 

Let us assume that the functions $u_{1}, u_{2}, \ldots, u_{k}$ form an orthonormal set in the space $L_2(\Omega_\beta)$, and these functions are the eigenfunctions of ${\mathcal A}_s^{\Omega_\beta}$ corresponding to the eigenvalues $\lambda_{1}({\mathcal A}_s^{\Omega_\beta}), \lambda_{2}({\mathcal A}_s^{\Omega_\beta}), \ldots, \lambda_{k}({\mathcal A}_s^{\Omega_\beta})$. Due to Lemma \ref{lem-q-form-monotony}, the functions $v_{j,\alpha} = \left(\frac{\sin\alpha}{\sin\beta}\right)^{1/2}u_{j}\circ T_{\alpha,\beta}$, $1\leqslant j\leqslant k$, belong to the space $\widetilde{H}^s(\Omega_\alpha)$, maintain their orthonormality in $L_2(\Omega_\alpha)$, and the following inequality is fulfilled
\begin{equation}
\label{thm2-form-ineq}
a_s^{\Omega_\alpha}[v_{j,\alpha}] < a_s^{\Omega_\beta}[u_{j}]. 
\end{equation}
Now, the existence of $k$ eigenvalues below the threshold and their monotonicity follows from \cite[Theorem 10.2.3]{BiSo}. Here we present a more straightforward way to derive this properties. We assign the orthonormal set of functions $E_j = \{v_{1,\alpha}, v_{2,\alpha},\ldots, v_{j,\alpha}\}$ in $L_2(\mathbb R^n)$ to each $j$ not greater than $k$. The $L_2(\mathbb R^n)$-norm of any convex combination of functions from $E_j$ is equal to one. Since the quadratic form $a_s$ is convex, the inequality \eqref{thm2-form-ineq} guaranties that the Rayleigh ratio of any function from the convex hull of $E_j $ is smaller than $\lambda_j({\mathcal A}_s^{\Omega_\beta})$. Hence, any subspace in $L_2(\mathbb R^n)$ with codimension $j-1$ has a nontrivial intersection with  the convex hull of $E_j$. Due to the max-min principle it follows that $\lambda_j({\mathcal A}_s^{\Omega_\alpha}) < \lambda_j({\mathcal A}_s^{\Omega_\beta})$ .
\end{proof}

\begin{remark}
    The first eigenvalue $\lambda_1({\mathcal A}_s^{\Omega_\alpha})$ tends to $\Lambda_\dagger$ as $\alpha\to\pi/2$. 
\end{remark}
\begin{proof}
Let $u$ be an eigenfunction normalized in $L_2(\Omega_\alpha)$ corresponding to $\lambda_1({\mathcal A}_s^{\Omega_\alpha})$. Using the notation provided in the proof of Lemma \ref{lem-q-form-monotony} we can write
 \begin{equation*}
    a_s^{\Omega_\alpha}[u] = C_s{\mathcal E}_s^{\Omega_{\pi/2}}(V_{\pi/2}) + \cos\alpha\cdot {\mathfrak r}(V_{\pi/2}).
\end{equation*}
Applying the following arithmetic–geometric mean inequality
\begin{equation*}
    2|\partial_x V_{\pi/2}(\pmb{\bf x},t)\partial_z V_{\pi/2}(\pmb{\bf x},t)| \leqslant
    |\nabla V_{\pi/2}(\pmb{\bf x},t)|^2,
\end{equation*}
since ${\mathfrak r}(V_{\pi/2}) < 0 $ we get an estimate
\begin{equation*}
     a_s^{\Omega_\alpha}[u] \geqslant C_s(1 - \cos\alpha){\mathcal E}_s^{\Omega_{\pi/2}}(V_{\pi/2}).
\end{equation*}
Since $\lambda_1({\mathcal A}_s^{\Omega_\alpha}) = a_s^{\Omega_\alpha}[u]$, Theorem \ref{thm-sp-in-tube} yields the following chain of inequalities
\begin{equation*}
    \Lambda_\dagger>\lambda_1({\mathcal A}_s^{\Omega_\alpha})\geqslant C_s(1 - \cos\alpha){\mathcal E}_s^{\Omega_{\pi/2}}(V_{\pi/2})\geqslant(1-\cos\alpha)\Lambda_\dagger.
\end{equation*}
Since $\cos\alpha\to0$ as $\alpha\to\pi/2$, the squeeze theorem yields the convergence of the first eigenvalue $\lambda_1({\mathcal A}_s^{\Omega_\alpha})$ to $\Lambda_\dagger$.
\end{proof}

\paragraph*{Acknowledgments.} 
The results were obtained under support of the Russian Science Foundation (RSF) grant 19-71-30002.

\printbibliography
\end{document}